\documentclass[final, 3p, times]{elsarticle}%
\usepackage{graphicx}
\usepackage{xcolor}
\usepackage{colortbl}
\usepackage{epstopdf}
\usepackage{amssymb}
\usepackage{amsthm}
\usepackage{amsmath}
\usepackage{color}
\usepackage{dsfont}
\usepackage{mathtools}
\usepackage{extarrows}
\usepackage{hyperref}
\usepackage{bm}
\usepackage{caption}
\usepackage{float}
\usepackage{verbatim}
\usepackage{epsfig}
\usepackage{multirow}
\usepackage{subfig}
\usepackage{algorithm}       
\usepackage{algpseudocode}   
\usepackage{threeparttable} 
\usepackage{booktabs} 
\usepackage[section]{placeins}

\newtheorem{assumption}{Assumption}[section]
\newtheorem{lem}{Lemma}[section]
\newtheorem{thm}{Theorem}[section]
\newtheorem{pro}{Proposition}[section]

\newtheorem{exm}{Example}[section]
\biboptions{compress}

\begin{document}
\begin{frontmatter}
\title{Random Greedy Fast Block Kaczmarz Method for Solving Large-Scale Nonlinear Systems}

\tnotetext[label1]{The research of Dongling Wang is supported in part by the National Natural Science Foundation of China under grants 12271463.
 \\ Declarations of interest: none.
}

\author[XTU]{Renjie Ding} 
\ead{drjmath@smail.xtu.edu.cn;}
\author[XTU]{Dongling Wang\corref{mycorrespondingauthor}}
\ead{wdymath@xtu.edu.cn; ORCID 0000-0001-8509-2837}
\cortext[mycorrespondingauthor]{Corresponding author. }

\address[XTU]{School of Mathematics and Computational Science, Xiangtan University, Xiangtan, Hunan 411105, China}

\begin{abstract}
To efficiently solve large-scale nonlinear systems, we propose a novel Random Greedy Fast Block Kaczmarz (RGFBK) method. This approach integrates the strengths of random and greedy strategies while avoiding the computationally expensive pseudoinversion of Jacobian submatrices, thus enabling efficient solutions for large-scale problems. Our theoretical analysis establishes that the proposed method achieves linear convergence in expectation, with its convergence rate’s upper bound determined by the \textit{stochastic greedy condition number} and the relaxation parameter.
Numerical experiments confirm that when the Jacobian matrix exhibits a favorable stochastic greedy condition number and an appropriate relaxation parameter is selected, the algorithm’s convergence is significantly accelerated. As a result, the proposed method outperforms other comparable algorithms in both efficiency and robustness.
\end{abstract}

\begin{keyword}
Nonlinear Kaczmarz method, Random and Greedy, Large-Scale Nonlinear Systems.
\end{keyword}

\end{frontmatter}

\section{Introduction}

Consider solving the following large-scale nonlinear system:
\begin{equation}\label{eq:problem}
F(x)=0,\forall x\in D
\end{equation}
where $F=(F_1,\cdots, F_m)^T:\mathbb{R}^n\to \mathbb{R}^m$ is continuously differentiable and continuous on the closed set $D\subset \mathbb{R}^n$. The solution of large-scale nonlinear systems finds extensive applications in scientific computation, such as optimization problems \cite{kelley1999iterative}, operator equations \cite{atkinson1992survey}, and machine learning \cite{chen2019homotopy}, etc. Therefore, developing a class of efficient algorithms for nonlinear systems is of crucial importance. 

Newton-type methods \cite{kelley2003solving} are classical approaches for solving \eqref{eq:problem}, but they typically require computing the inverse of the Jacobian matrix, which is highly expensive. Over the past few decades, the Kaczmarz method \cite{strohmer2009randomized, bai2018relaxed, du2020randomized, zhang2023randomized}, as an efficient method for solving linear systems, has been well-developed. Recently, Kaczmarz-type methods for solving nonlinear systems \cite{wang2022nonlinear, gao2025convergence, xiao2025fast} have been studied, with their iterative formula given as follows:
\begin{equation}\label{eq:base update}
	x_{k+1} = x_{k} - \frac{F_{i_k}(x_k)}{\nabla F_{i_k}(x_k)}\nabla F_{i_k}(x_k),\, \text{for} \, k=1,2\cdots, 
\end{equation}
where $i_k\in [m]:=\{1,\cdots,m\}$ is determined by a cyclic or random strategy. From \eqref{eq:base update}, it can be seen that such methods are simple and easy to implement; however, their convergence rate may be relatively slow. To further improve the performance of the methods, several block Kaczmarz methods have been proposed \cite{zhang2023maximum, zhangli2024greedy}, which are defined as follows:
\begin{equation}\label{eq:block update}
	x_{k+1} = x_{k} - (F_{\mathcal{I}_k}(x_k))^{\dagger} F_{\mathcal{I}_k}(x_k),\, \text{for} \,\, k=1,2\cdots, 
\end{equation}
where \( F_{\mathcal{I}_k}'(x_k)^\dagger \) denotes the Moore-Penrose pseudoinverse of the row submatrix of the Jacobian matrix defined by the index \( \mathcal{I}_k \subset [m] \). 
Generally, these index set $\mathcal{I}_k$ determination strategies fall into two categories: one random-based, the other greedy-based. For instance, to prioritize eliminating larger residuals, RB-CNK \cite{zhang2024greedy} and RBWNK\cite{ye2024residual} select index sets based on the following greedy strategy:
\begin{equation}\label{eq:past Ik}
  \mathcal{I}_{k}=\left\{i\mid\left|f_{i}(x_{k})\right|^{2}\geq\delta_{k}\|f(x_{k})\|_{2}^{2}\right\},
\end{equation}
where $\delta_{k}=\frac{1}{2}\left(\frac{\max_{i\in[m]}\left|f_{i}(x_{k})\right|^{2}}{\left\|f(x_{k})\right\|_{2}^{2}}+\frac{1}{m}\right)$. To reduce the cost of determining index sets, RB-BSNK and DB-BSNK \cite{zhangli2024greedy} introduce subsampling strategies. Specifically, they first randomly select a set uniformly, then determine the final index set \( \mathcal{I}_k \) based on the maximum residual rule or maximum distance rule. Despite their effectiveness, their strategies cannot pre-control the size of the set, which may lead to extreme cases. In addition, to overcome the pseudoinverse computing cost in \eqref{eq:block update}, several pseudoinverse-free schemes have been proposed, yet all rely on greedy strategies for index set determination, and the aforementioned issues still exist. More relevant papers can be found in \cite{tan2024nonlinear, xiao2024averaging, zhang2024greedy}.

In this paper, in order to overcome the aforementioned difficulties and design an efficient block scheme suitable for solving large-scale problems,  we introduce both random and greedy strategies to determine the index set for each iteration.  Additionally, we derive a pseudoinverse-free update formula by solving the weighted linear residual optimization problem \eqref{eq:weighted linear res}. Consequently, compared with Kaczmarz-type methods based on pseudoinverse computation, our method achieves higher computational efficiency. In contrast to other pseudoinverse-free Kaczmarz-type methods, our strategy reduces the cost associated with determining the index set while inheriting the advantages of both random and greedy strategies.

The structure of the remaining part of the article is as follows: In Section 2, we establish the Random Greedy Fast Block Kaczmarz method and its convergence theory. In Section 3, we verify the effectiveness of the algorithm through numerical experiments. Section 4 presents the conclusions.

\emph{Notation:} Throughout the paper, for \( x \in \mathbb{R}^n \), the standard Euclidean norm is defined as \( \|x\|_2 = \sqrt{x^\mathsf{T}x} \). For matrix \( A \in \mathbb{R}^{m \times n} \), \( \text{range}(A) \), \( \sigma_{\min}(A) \), \( \sigma_{\max}(A) \) are its column space, nonzero minimum, and maximum singular values, respectively. For any positive integer \( m \), we define the index set \( [m] \) as \( [m] := \{1, 2, \ldots, m\} \). 
Given an index set \( J \subset [m] \), where \( |J| \) denotes the cardinality of \( J \), we define \( x_J \in \mathbb{R}^{|J|} \) as the subvector of \( x \) containing entries with indices corresponding to the rows in \( J \); likewise, \( A_J \in \mathbb{R}^{|J| \times n} \) denotes the row submatrix of \( A \) whose rows are indexed by \( J \).  We introduce the probability space \(([m], \mathcal{F}, \mathbb{P})\), where \(\mathcal{F} \subset 2^{[m]}\) (the power set of \([m]\)) serves as a \(\sigma\)-algebra. Let \(J\) be a random set-valued mapping with \(J \subseteq [m]\) (taking values in \(2^{[m]}\)), whose realizations are governed by the probability distribution \(\mathbb{P}(J)\). Thus, we define the \textit{stochastic greedy condition number} as follows:  
\begin{equation}\label{sgc}
	\kappa_{\alpha,\, \beta}(A):=\max_{I\in\mathcal{I}^{\alpha,\, \beta}}\kappa(A_I)=\max_{I\in\mathcal{I}^{\alpha,\, \beta}}\frac{\sigma_{\max}(A_I)}{\sigma_{\min}(A_I)},
\end{equation}
where the meaning of \( \mathcal{I}^{\alpha,\,\beta} \) will be explained subsequently.  

\section{The RGFBK method}
We now introduce the RGFBK method from two aspects: the index set determination rule and the update scheme. 
First, we introduce the random greedy strategy to determine the index set \( \mathcal{I}_k \).

Given \( \alpha, \beta \in \mathbb{N}_+ \) with \( \alpha \in [1, m] \) and \( \beta \in [1, \alpha] \), suppose the iterate \( x_k \in \mathbb{R}^n \) and the residual \( |F(x_k)| \in \mathbb{R}^m \). We draw \( \alpha \) distinct indices uniformly at random from \([m]\) to form random set-valued mapping \( J \). For \( J \), there are \( \binom{m}{\alpha} \) possible realizations, each with probability \( \mathbb{P}(J = J_k) = 1/\binom{m}{\alpha} \) for any specific realization \( J_k \subset [m] \) of \( J \), where \( |J_k| = \alpha \). From \( |F_{J_k}(x_k)| \), we select the indices of the \( \beta \) largest values to form a deterministic set \( \mathcal{I}_k(J_k) \), and the randomness of \( \mathcal{I}_k \) is inherited solely from \( J \), so \( \mathbb{P}(\mathcal{I}_k = \mathcal{I}_k(J_k)) =\mathbb{P}(J = J_k) = 1/\binom{m}{\alpha}\). We denote the sample space of \( J \) as \( \Omega^\alpha = \{ J_k \subset [m] \mid |J| = \alpha \} \), and the set of all possible realizations of \( \mathcal{I}_k \) as 
\[
\mathcal{I}^{\alpha,\,\beta} = \left\{ \mathcal{I}_k(J_k) \mid J_k \in \Omega^\alpha \right\}. 
\]
Thus, through uniform random sampling of \( J \) followed by the greedy selection rule, each realization \( \mathcal{I}_k(J_k) \in \mathcal{I}^{\alpha,\,\beta} \) of the random set \( \mathcal{I}_k \) corresponds to a unique index set. Thus, our subsampling strategy combines the strengths of randomness and greediness, allows the block size to be predetermined, and suits large-scale problems better.

To design an efficient block method, we consider minimizing the weighted linear residual optimization problem:
\begin{equation}\label{eq:weighted linear res}
	\begin{aligned}
		\min_{x\in\mathbb{R}^{n}} & \left\| \omega_k^T (F_{\mathcal{I}_k}(x_k)+F^{\prime}_{\mathcal{I}_k}(x_k)(x-x_k)) \right\|_2^2,
	\end{aligned}
\end{equation}
where \( \omega_k \in \mathbb{R}^{|\mathcal{I}_k|} \) denotes a weight vector, and \( \mathcal{I}_k \in \mathcal{I}^{\alpha,\, \beta} \) is a realization of the random variable \( \mathcal{I}_k \) from the aforementioned random greedy rule.  
We select appropriate weights to reflect the importance of different residual components. In particular, we can choose \( \omega_k = F_{\mathcal{I}_k}(x_k) \). By solving the weighted linear residual optimization problem and introducing the relaxation parameter \( \gamma \), we obtain the following update formula:
\begin{equation}\label{eq:update}
	\begin{aligned}
		x_{k+1} = x_k - \gamma \frac{\omega_k^TF_{\mathcal{I}_k}(x_k)}{\|F_{\mathcal{I}_k}^{\prime}(x_k)^T\omega_k\|}F_{\mathcal{I}_k}^{\prime}(x_k)^T\omega_k,
	\end{aligned}
\end{equation}

Summarizing the above ideas, we obtain Algorithm \ref{alg:RGFBK} as follows:
\begin{algorithm}
	\caption{Random Greedy Fast Block Kaczmarz Method for Solving Large-Scale Nonlinear Systems}
	\label{alg:RGFBK}
	\begin{algorithmic}[1]
		\State \textbf{Input:} $x_0$, parameter $\alpha$, $\beta$ and $\gamma\in(0, 2)$, $k=0$, $r_0=F(x_0)$
		\While {The stopping criterion is not satisfied}
		\State Uniformly randomly select \( \alpha \) samples from \([m]\) to form control set \( \Omega_k \).

		\State From \( \Omega_k \), select the \( \beta \) samples with the largest residuals to form \( \mathcal{I}_k \).
		
		\State The weight vector \( \omega_k \) is computed according to the given rule, e.g., \( \omega_k = F_{\mathcal{I}_k}(x_k) \).  
		
		\State Update  $x_{k} \leftarrow x_k - \gamma \frac{\omega_k^TF_{\mathcal{I}_k}(x_k)}{\|F_{\mathcal{I}_k}^{\prime}(x_k)^T\omega_k\|}F_{\mathcal{I}_k}^{\prime}(x_k)^T\omega_k$
		\State Compute residual $r_k=F(x_k)$
		\EndWhile
		\State \textbf{Output:} $x_{k}$
	\end{algorithmic}
\end{algorithm}

Updated formula \eqref{eq:update} can also be understood as a special variant of the sketching Newton method \cite{yuan2022sketched}. When the relaxation parameter \(\gamma = 1\), the sketch matrix $S_k\in \mathbb{R}^{m}$ is exactly $S_k^{(i)} = 
	\begin{cases} 
		F_{i}(x_k) & \text{if} \quad i \in \mathcal{I}_k , \\
		0 \quad & \text{if} \quad i \notin \mathcal{I}_k.
	\end{cases}$.
Thus, from the perspective of sketching, more choices for weight vectors are possible; for instance, we can select Gaussian vectors as the weight vectors.

We now proceed to discuss the convergence theory when the weight vector is taken as the residual vector. To analyze the convergence of Algorithm \ref{alg:RGFBK}, we first present the following basic assumptions.

\begin{assumption}\label{as:asc}\cite{haltmeier2007kaczmarz}
	The following conditions hold:
	\begin{itemize}
		\item[(i)] 
	For each \(i \in \{1, 2, \ldots, m\}\) and any \(x_1, x_2 \in \mathbb{R}^n\), there exists some \(\eta_i \in [0, \eta)\) where \(\eta = \max_{1 \leq i \leq m} \eta_i < \frac{1}{2}\), such that\(
	|F_i(x_1) - F_i(x_2) - \nabla F_i(x_1)(x_1 - x_2)| \leq \eta_i |F_i(x_1) - F_i(x_2)|.
	\)
	
	In such cases, the function \(F(x): \mathbb{R}^n \to \mathbb{R}^m\) is said to satisfy the local tangential cone condition.
\item[(ii)] 
For $x\in D, F^{\prime}(x)$ is row bounded below and full column rank matrix, i.e., there exits a positive constant $\varepsilon>0$ such that  $\|F_{i}^{\prime}(x)\|\geq \varepsilon$ for any $i\in [m]$. 
	\end{itemize}
\end{assumption}
Next, we present some properties and lemmas required for the convergence analysis.

\begin{pro}\label{pro}\cite{xiao2024averaging}
	Suppose the function \( F \) meets the local tangential cone condition. Then, given any \( x_1, x_2 \in \mathbb{R}^n \) and an index subset \( I \subset \{1, 2, \ldots, m\} \), the following inequalities hold:
	
	\[
	\left\lVert F_I(x_1) - F_I(x_2) - F_I'(x_1)(x_1 - x_2) \right\rVert_2^2 \leq \eta^2 \left\lVert F_I(x_1) - F_I(x_2) \right\rVert_2^2 \tag{3.1}
	\]

	\[
	\left\lVert F_I(x_1) - F_I(x_2) \right\rVert_2^2 \geq \frac{1}{1 + \eta^2} \left\lVert F_I'(x_1)(x_1 - x_2) \right\rVert_2^2 \tag{3.2}
	\] 
\end{pro}

\begin{lem}\label{le:key lem}
	Suppose that  $F(x)$ satisfies Assumption \ref{as:asc} and there exists $x_*\in D$ such that $F(x_*)=0$. The sequence $\{x_k\}_{k\geq0}$ generated by \( x_{k+1} = x_k - \gamma \frac{\omega_k^T F_{\mathcal{I}_k}(x_k)}{\|F_{\mathcal{I}_k}^{\prime}(x_k)^T \omega_k\|} F_{\mathcal{I}_k}^{\prime}(x_k)^T \omega_k \) has the estimate:
	\begin{equation}
			\|x_{k+1}-x_*\|_2^2 \leq \|x_k - x_*\|_2^2 - [2\gamma(1-\eta)-\gamma^2] \frac{|\omega_k^T F_{\mathcal{I}_k}(x_k)|\cdot\|F_{\mathcal{I}_k}(x_k)\|_2}{{\sigma_{\max}^2(F'_{\mathcal{I}_k}(x_k))}\cdot\|\omega_k\|_2}.
		\end{equation}
In particular, when \(\omega_k = F_{\mathcal{I}_k}(x_k)\), we have
\begin{equation}
	\|x_{k+1}-x_*\|_2^2 \leq \|x_k - x_*\|_2^2 - [2\gamma(1-\eta)-\gamma^2] \frac{\|F_{\mathcal{I}_k}(x_k)\|_2^2}{{\sigma_{\max}^2(F'_{\mathcal{I}_k}(x_k))}}.
\end{equation}
\end{lem}
\begin{proof}
	Substituting the expression of the iteration formula, we have
	
	$\begin{aligned}
		\left\|x_{k+1}-x_{*}\right\|_{2}^{2} &  =\left\|x_{k}-x_{*}\right\|_{2}^{2}+\gamma^2\left\|\frac{\omega_k^T F_{I_{k}}(x_{k})}{\left\|F_{I_{k}}^{\prime}(x_{k})^T\omega_k\right\|_{2}^{2}}F_{I_{k}}^{\prime}(x_{k})^T\omega_k\right\|_{2}^{2}-2\gamma\left\langle\frac{\omega_k^T F_{I_{k}}(x_{k})}{\left\|F_{I_{k}}^{\prime}(x_{k})^T\omega_k\right\|_{2}^{2}}F_{I_{k}}^{\prime}(x_{k})^T\omega_k,x_{k}-x_{*}\right\rangle \\
		& =\left\|x_{k}-x_{*}\right\|_{2}^{2}+\gamma^2\left\|\frac{\omega_k^T F_{I_{k}}(x_{k})}{\left\|F_{I_{k}}^{\prime}(x_{k})^T\omega_k\right\|_{2}^{2}}F_{I_{k}}^{\prime}(x_{k})^T\omega_k\right\|_{2}^{2}+2\gamma\frac{\omega_k^T F_{I_{k}}(x_{k})}{\left\|(F_{I_{k}}^{\prime}(x_{k}))^{T}\omega_k\right\|_{2}^{2}}\omega_k^{T}(F_{I_{k}}(x_{k})-F_{I_{k}}(x_{*})-F_{I_{k}}^{\prime}(x_{k})(x_{k}-x_{*})) \\
		& -2\gamma\frac{\left\|\omega_k^T F_{I_{k}}(x_{k})\right\|_2^2}{\left\|F_{I_{k}}^{\prime}(x_{k})^{T}\omega_k\right\|_{2}^{2}} .\\
	\end{aligned}$
By the Cauchy-Schwarz inequality and Property \ref{pro}, we can obtain

$\begin{aligned}
	 \left\|x_{k+1}-x_{*}\right\|_{2}^{2}& \leq \left\|x_{k}-x_{*}\right\|_{2}^{2}+2\gamma\frac{|\omega_k^T F_{I_{k}}(x_{k})|\cdot\|\omega_k\|_2}{\left\|(F_{I_{k}}^{\prime}(x_{k}))^{T}\omega_k\right\|_{2}^{2}}\|F_{I_{k}}(x_{k})-F_{I_{k}}(x_{*})-F_{I_{k}}^{\prime}(x_{k})(x_{k}-x_{*})\|_2  +(\gamma^2-2\gamma)\frac{|\omega_k^T F_{I_{k}}(x_{k})|^2}{\left\|F_{I_{k}}^{\prime}(x_{k})^T\omega_k\right\|_{2}^{2}}\\
	&\leq \left\|x_{k}-x_{*}\right\|_{2}^{2}+2\gamma\frac{|\omega_k^T F_{I_{k}}(x_{k})|\cdot\|\omega_k\|_2}{\left\|(F_{I_{k}}^{\prime}(x_{k}))^{T}\omega_k\right\|_{2}^{2}}\cdot\eta\|F_{\mathcal{I}_k}(x_k)\|  +(\gamma^2-2\gamma)\frac{|\omega_k^T F_{I_{k}}(x_{k})|^2}{\left\|F_{I_{k}}^{\prime}(x_{k})^T\omega_k\right\|_{2}^{2}}\\
	&\leq  \left\|x_{k}-x_{*}\right\|_{2}^{2}-[2\gamma\cdot(1-\eta)-\gamma^2]\frac{|\omega_k^T F_{I_{k}}(x_{k})|\cdot\|\omega_k \|_2\cdot\| F_{I_{k}}(x_{k})\|_2}{\left\|F_{I_{k}}^{\prime}(x_{k})^T\omega_k\right\|_{2}^{2}}\\
	&\leq \|x_k - x_*\|_2^2 - [2\gamma(1-\eta)-\gamma^2] \frac{|\omega_k^T F_{\mathcal{I}_k}(x_k)|\cdot\|F_{\mathcal{I}_k}(x_k)\|_2}{{\sigma_{\max}^2(F'_{\mathcal{I}_k}(x_k))}\cdot\|\omega_k\|_2}.
\end{aligned}$
In particular, when \(\omega_k = F_{\mathcal{I}_k}(x_k)\), we have

$\begin{aligned}
	\left\|x_{k+1}-x_{*}\right\|_{2}^{2}
	&\leq \|x_k - x_*\|_2^2 - [2\gamma(1-\eta)-\gamma^2] \frac{|\omega_k^T F_{\mathcal{I}_k}(x_k)|\cdot\|F_{\mathcal{I}_k}(x_k)\|_2}{{\sigma_{\max}^2(F'_{\mathcal{I}_k}(x_k))}\cdot\|\omega_k\|_2}\\
	&
	 \leq \|x_k - x_*\|_2^2 - [2\gamma(1-\eta)-\gamma^2] \frac{\|F_{\mathcal{I}_k}(x_k)\|_2^2}{{\sigma_{\max}^2(F'_{\mathcal{I}_k}(x_k))}}.
\end{aligned}$

This completes the proof.
\end{proof}

\begin{thm}
	Suppose that  $F(x)$ satisfies Assumption \ref{as:asc} and there exists $x_*\in D$ such that $F(x_*)=0$ and put the residual vector \( F(x_k) \) as the weight vector \( \omega_k \) for each iteration  with \( \gamma \in (0, 2(1-\eta)) \).
	 Then the sequence \( \{x_k\} \)  generated by Algorithm \ref{alg:RGFBK} converges to \( x_* \) in expectation satisfying that
	\begin{equation}\label{thm}
	\mathbb{E}\left[\|x_{k+1}-x_*\|_2^2\right] \leq \left( 1 - \frac{2\gamma(1-\eta)-\gamma^2}{(1+\eta^2)\kappa^2_{\alpha,\, \beta}(F(x_k))} \right) \mathbb{E}\left[\|x_{k}-x_*\|_2^2\right] 
	\end{equation}
	where the stochastic greedy condition number  \( \kappa_{\alpha,\, \beta}(F(x_k)) \) is defined in \eqref{sgc}.
\end{thm}	
\begin{proof}
	By Lemma \ref{le:key lem}, we have
	\[
		\left\|x_{k+1}-x_{*}\right\|_{2}^{2}
		\leq \|x_k - x_*\|_2^2 - [2\gamma(1-\eta)-\gamma^2] \frac{\|F_{\mathcal{I}_k}(x_k)\|_2^2}{{\sigma_{\max}^2(F'_{\mathcal{I}_k}(x_k))}}.
\]
	
Computing the conditional expectation for both sides of the aforementioned inequality yields

	$\begin{aligned}
	\mathbb{E}_k[\| x_{k + 1} - x_* \|_2^2]
	&\leq \|x_k - x_*\|_2^2 - [2\gamma(1-\eta)-\gamma^2] \mathbb{E}_k\left[\frac{\|F_{\mathcal{I}_k}(x_k)\|_2^2}{{\sigma_{\max}^2(F'_{\mathcal{I}_k}(x_k))}}\right]\\
	&= \|x_k - x_*\|_2^2 -  [2\gamma(1-\eta)-\gamma^2]\sum_{\mathcal{I}_k\in \mathcal{I}^{\alpha,\,\beta}}\frac{\|F_{\mathcal{I}_k}(x_k)-F_{\mathcal{I}_k}(x_*)\|_2^2}{\binom{m}{\alpha}\cdot\sigma_{\max}^2(F'_{\mathcal{I}_k}(x_k))}\\
	&\leq \|x_k - x_*\|_2^2 -  \frac{[2\gamma(1-\eta)-\gamma^2]}{{(1+\eta^2)\cdot\binom{m}{\alpha}}}\sum_{\mathcal{I}_k\in \mathcal{I}^{\alpha,\,\beta}}\frac{\|F_{\mathcal{I}_k}^{\prime}(x_k)(x_k-x_*)\|_2^2}{\sigma_{\max}^2(F'_{\mathcal{I}_k}(x_k))}\\
	&\leq \|x_k - x_*\|_2^2 -  \frac{[2\gamma(1-\eta)-\gamma^2]}{{(1+\eta^2)\cdot\binom{m}{\alpha}}}\sum_{\mathcal{I}_k\in \mathcal{I}^{\alpha,\,\beta}}\frac{\sigma_{\max}^2(F'_{\mathcal{I}_k}(x_k))\cdot\|x_k-x_*\|_2^2}{\sigma_{\max}^2(F'_{\mathcal{I}_k}(x_k))}\\
	&\leq \|x_k - x_*\|_2^2 -  \frac{[2\gamma(1-\eta)-\gamma^2]}{{(1+\eta^2)\cdot\binom{m}{\alpha}}}\sum_{\mathcal{I}_k\in \mathcal{I}^{\alpha,\,\beta}}\frac{\|x_k-x_*\|_2^2}{\kappa^2(F'_{\mathcal{I}_k}(x_k))}\\
	&\leq \left( 1 - \frac{2\gamma(1-\eta)-\gamma^2}{(1+\eta^2)\kappa^2_{\alpha,\, \beta}(F(x_k))} \right) \| x_k - x_* \|_2^2.
\end{aligned}$

By computing the full expectation for both sides of the above inequality, we attain the desired result \eqref{thm}.
\end{proof}

	From Theorem \eqref{thm}, the upper bound on the convergence rate of the RGFBK method is $	\rho_k\left(\kappa^2_{\alpha,\, \beta}(F(x_k)), \,\gamma\right) =1 - \frac{2\gamma(1-\eta)-\gamma^2}{(1+\eta^2)\kappa^2_{\alpha,\, \beta}(F(x_k))} $. We observe that it is determined by the stochastic greedy condition number \(\kappa^2_{\alpha,\, \beta}(F(x_k))\) and the relaxation parameter \(\gamma\). When the relaxation parameter \(\gamma = 1 - \eta\), we have 
	\[
	\rho_k\left(\kappa^2_{\alpha,\, \beta}(F(x_k))\right) := \min_{\gamma \in (0,\, 2(1-\eta))} \frac{2\gamma(1-\eta) - \gamma^2}{(1+\eta^2)\kappa^2_{\alpha,\, \beta}(F(x_k))} = \frac{1-2\eta+\eta^2}{(1+\eta^2)\kappa^2_{\alpha,\, \beta}(F(x_k))}.
	\]  
	

\section{Numerical experiments}
We confirm the effectiveness and advantages of the RGFBK method through numerical examples and compare it with the following algorithms:

	The Maximum Residual-Based Block Sampling Nonlinear Kaczmarz (MR-BSNK) \cite{zhangli2024greedy}.
	
	The Maximum distance-Based Block Sampling Nonlinear Kaczmarz (MD-BSNK) \cite{zhangli2024greedy}.
	
	The residual-based block capped nonlinear Kaczmarz (RB-CNK) \cite{zhang2024greedy}.
	
	The residual-based weighted nonlinear Kaczmarz method (RBWNK) \cite{ye2024residual}.


In all numerical experiments, we adopt the following stopping criterion: \( \| r_k \|_2 \leq \text{RES} \) or the number of iterations (IT) exceeds \( 10^5 \), where \( \text{RES} \) is defined by  
$
\text{RES} = 10^{-6} + 10^{-8} \| r_0 \|_2 
$
as proposed in \cite{kelley2003solving}. The performance of different methods is evaluated using IT (number of iterations) and the elapsed CPU time from algorithm initiation until the stopping criterion is met. 
For the RGFBK method, we set \( \alpha = \lfloor 0.75m \rfloor \), \( \beta = \lfloor 0.5\alpha \rfloor \), and \( \gamma = 1.2 \). For other methods, appropriate parameters are selected to yield the best numerical performance. The Methods that do not converge are excluded from the figures. 

\begin{exm}
The Chandrasekhar H-equation \cite{babajee2010analysis}.
\end{exm}
	\begin{equation}
	F(H)(\mu) = H(\mu) - \left(1 - \frac{c}{2}\int_0^1 \frac{\mu H(\nu)}{\mu + \nu} d\nu\right)^{-1} = 0.
\end{equation}
Rooted in radiative transfer theory, this equation involves integrals approximated using the composite midpoint rule:
$
	\int_0^1 f(\mu) d\mu \approx \frac{1}{N}\sum_{j=1}^N f(\mu_j),
$
with \(\mu_i = (i - 1/2)/N\) for \(1 \leq i \leq N\), where \(N\) is the number of integration nodes. This leads to the discrete form:
\begin{equation}\label{eq:dis}
	F_i(x) = x_i - \left(1 - \frac{c}{2N}\sum_{j=1}^N \frac{\mu_i x_j}{\mu_i + \mu_j}\right)^{-1}.
\end{equation}
Taking $c=0.9$, the \eqref{eq:dis} is the $N$-dimensional nonlinear system we need to solve, and its Jacobian matrix is dense. We set the initial value as $x_0=(0,\cdots, 0)\in \mathbb{R}^m$.
\begin{figure}[!htb]
	\centering
	\subfloat
	\centering
	\includegraphics[width=0.35\textwidth]{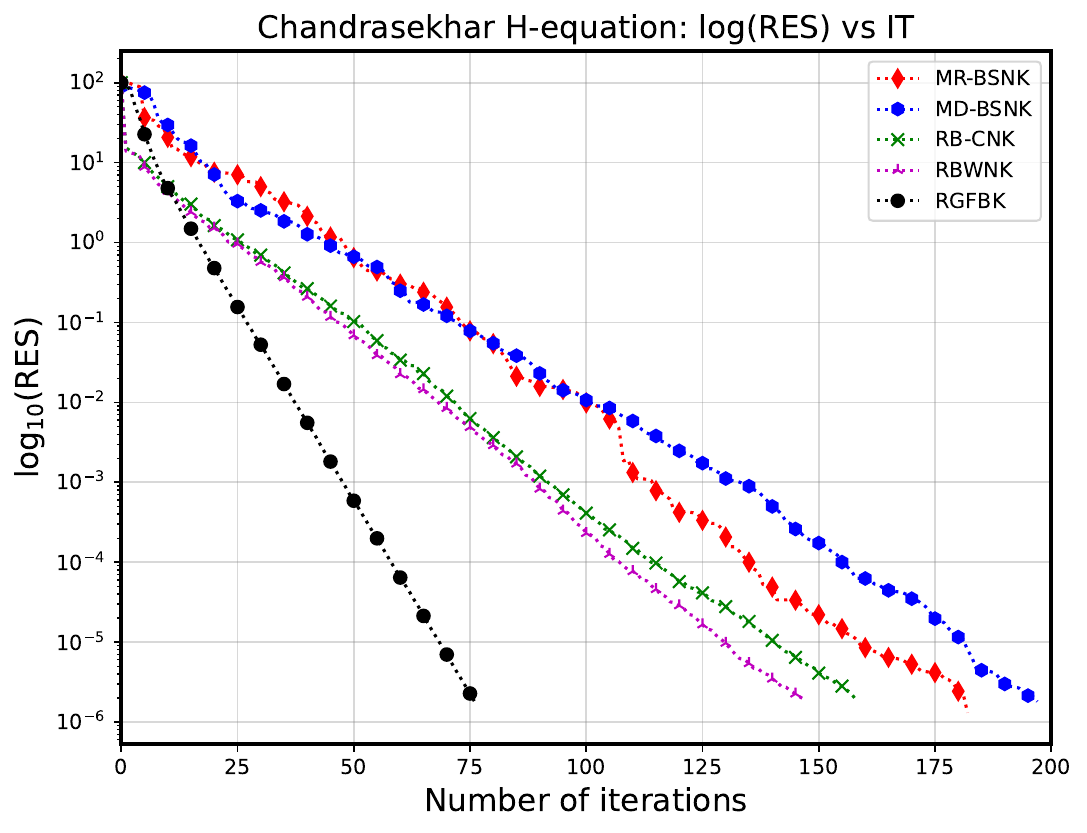}
	\hspace{0.1cm}
	\subfloat
	\centering
	\includegraphics[width=0.35\textwidth]{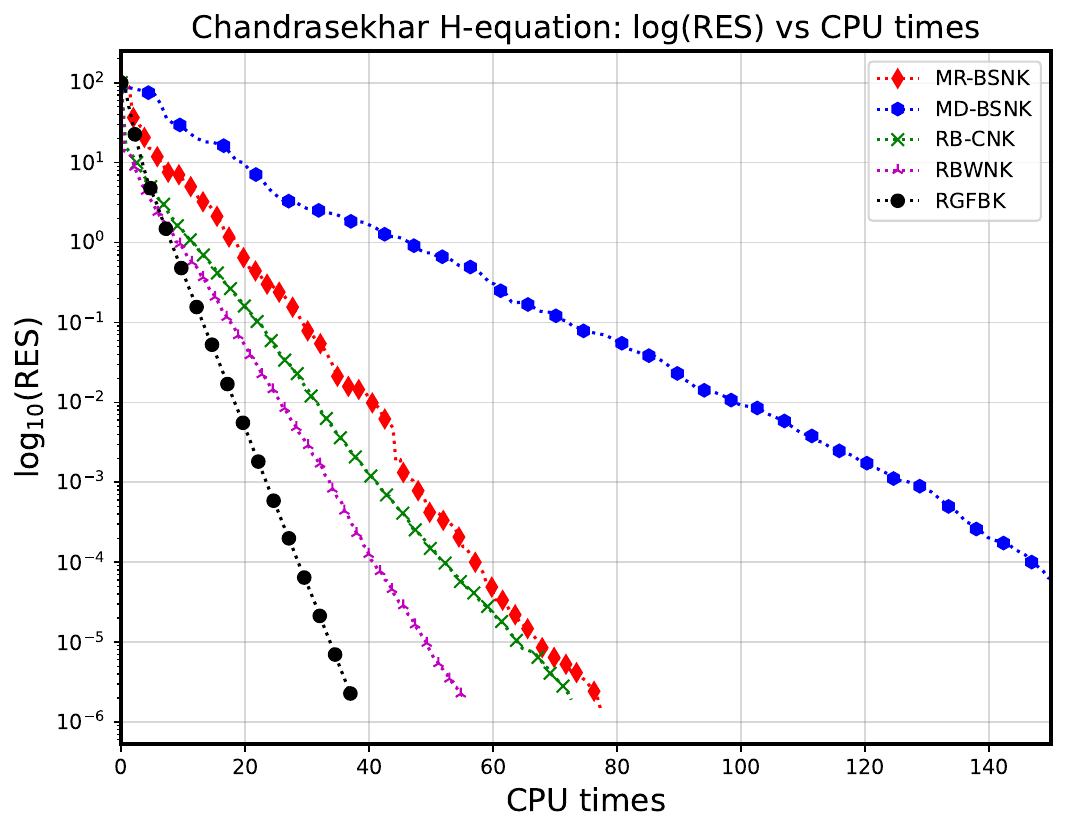}
	\hspace{0.1cm}

			\caption{Numerical results for the Chandrasekhar H-equations, $m=10^4$}
	\label{fig:ch}
\end{figure}
\begin{figure}[!htb]
	\centering
	\subfloat
	\centering
	\includegraphics[width=0.35\textwidth]{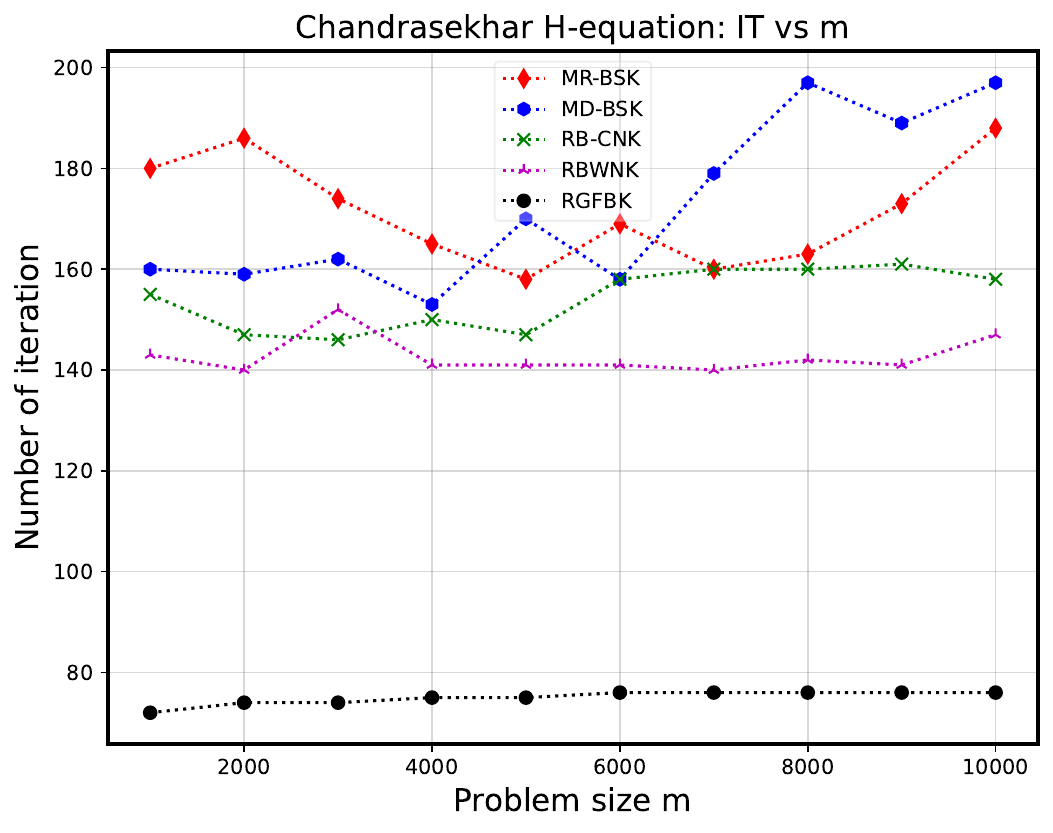}
	\hspace{0.1cm}
	\subfloat
	\centering
	\includegraphics[width=0.35\textwidth]{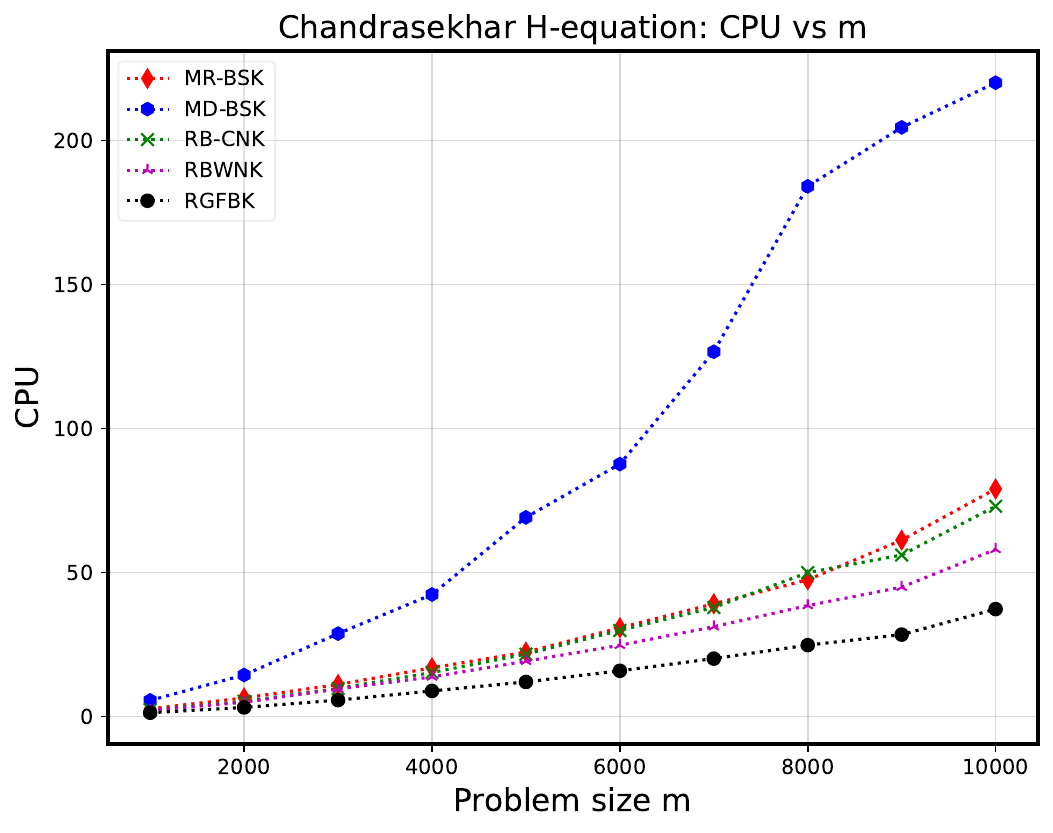}
	\hspace{0.1cm}
	
	\caption{Evolution of IT and CPU with the change of problem scale m, Chandrasekhar H-equations}
	\label{fig:chm}
\end{figure}
\begin{table}[!h]
	\caption{\label{table ch} IT and CPU of different methods for Chandrasekhar H-equation.}
	\begin{center}
		\begin{tabular}{l c cc cc cc cc cc cc}
			\toprule 
			& \multicolumn{1}{c}{problem size $m$} & \multicolumn{2}{c}{2000} & \multicolumn{2}{c}{4000}& \multicolumn{2}{c}{6000}& \multicolumn{2}{c}{8000}& \multicolumn{2}{c}{10000} \\
			\cmidrule(lr){3-4} \cmidrule(lr){5-6} \cmidrule(lr){7-8} \cmidrule(lr){9-10} \cmidrule(lr){11-12}
			& mathod & \multicolumn{1}{c}{IT} & \multicolumn{1}{c}{CPU} & \multicolumn{1}{c}{IT} & \multicolumn{1}{c}{CPU} & \multicolumn{1}{c}{IT} & \multicolumn{1}{c}{CPU}  & \multicolumn{1}{c}{IT} & \multicolumn{1}{c}{CPU} & \multicolumn{1}{c}{IT} & \multicolumn{1}{c}{CPU}\\ \midrule
			& MR-BSNK & 183 & 6.2154 & 166 & 17.556 & 172 & 30.798 & 162 & 49.850 & 182 & 77.693 \\
			& MD-BSNK & 159 & 16.425 & 158 & 46.264 & 158 & 96.335 & 195 & 183.72 & 197 & 216.53 \\
			& RB-CNK & 147 & 5.3850 & 150 & 15.581 & 158 & 28.765 & 160 & 50.7732 &158 & 73.854 \\
			& RBWNK & 140 & 5.0289 & 141 & 13.951 & 141 & 23.665 & 142 & 37.168 & 147 & 56.235 \\
			& RGFBK & 74 & 3.0256 & 75 & 8.1215 & 75 & 15.337 & 76 & 23.586 & 76 &35.524 \\ \bottomrule
		\end{tabular}
	\end{center}
\end{table}

From Figure \ref{fig:ch}, we can observe that the proposed RGFBK method outperforms all other block methods in terms of IT and CPU.
In Figure \ref{fig:chm} and Table \ref{table ch}, we present the variations in IT and CPU of different methods with respect to the problem size $m$. It can be observed that the proposed RGFBK method still outperforms other methods in terms of IT and CPU. Moreover, Figure \ref{fig:chm} shows that its growth curve is the flattest, confirming the effectiveness of the adopted random greedy strategy and the superior suitability of RGFBK for large-scale problems.

\begin{exm} 
	The  Broyden Tridiagonal Function \cite{lukvsan1994inexact}.
\end{exm}
	\begin{equation}
	\begin{cases}
		&F_{k}(x) =x_k(0.5x_k-3)+2x_{k+1}-1, \quad k=1,\notag\\
		&F_{k}(x)=x_k(0.5x_k-3)+x_{k-1}+2x_{k+1}-1, \quad 1<k<m,\\
		&F_{k}(x) =x_k(0.5x_k-3)-1+x_{k-1},k=m,
		\quad 
		k=1,\dots, m. 
	\end{cases}  
\end{equation}
We set the initial value as \( x_0 = (-1, \cdots, -1) \in \mathbb{R}^m \).
Figure \ref{fig:btm} shows the RGFBK method outperforms others in IT and CPU time. As \( m \) grows, its IT and CPU time increase very slow. 
Conversely, other methods’ performance deteriorates: the MD-BSNK method fails to converge across all problem sizes, while RBWNK and RB-CNK fail to converge as \( m \) increases. Thus, RGFBK suits large-scale problems much better.

From Figure \ref{fig:btp}, we observe how \( \alpha \) and \( \beta \) affect RGFBK’s performance: (i) For fixed \( \alpha \): IT decreases with \( \beta \) (except extreme full  sample cases); CPU first decreases then increases with \( \beta \) (excluding \( \alpha = 1000 \)). (ii) For fixed \( \beta \): IT decreases with \( \alpha \) (except \( \beta = 0.4\alpha \), where \( \alpha = 800 \) yields optimal IT); CPU attains optimality with properly chosen \( \alpha \), while extreme cases are costlier.  
The optimal result occurs at \( \alpha = 600 \), \( \beta = 0.4\alpha \), reflecting a trade-off between randomness and greediness. Here, the block size per iteration is \( \beta = 0.4\alpha = 240 \), well controlling memory usage and computational cost.

\begin{figure}[!htb]
	\centering
	\subfloat
	\centering
	\includegraphics[width=0.35\textwidth]{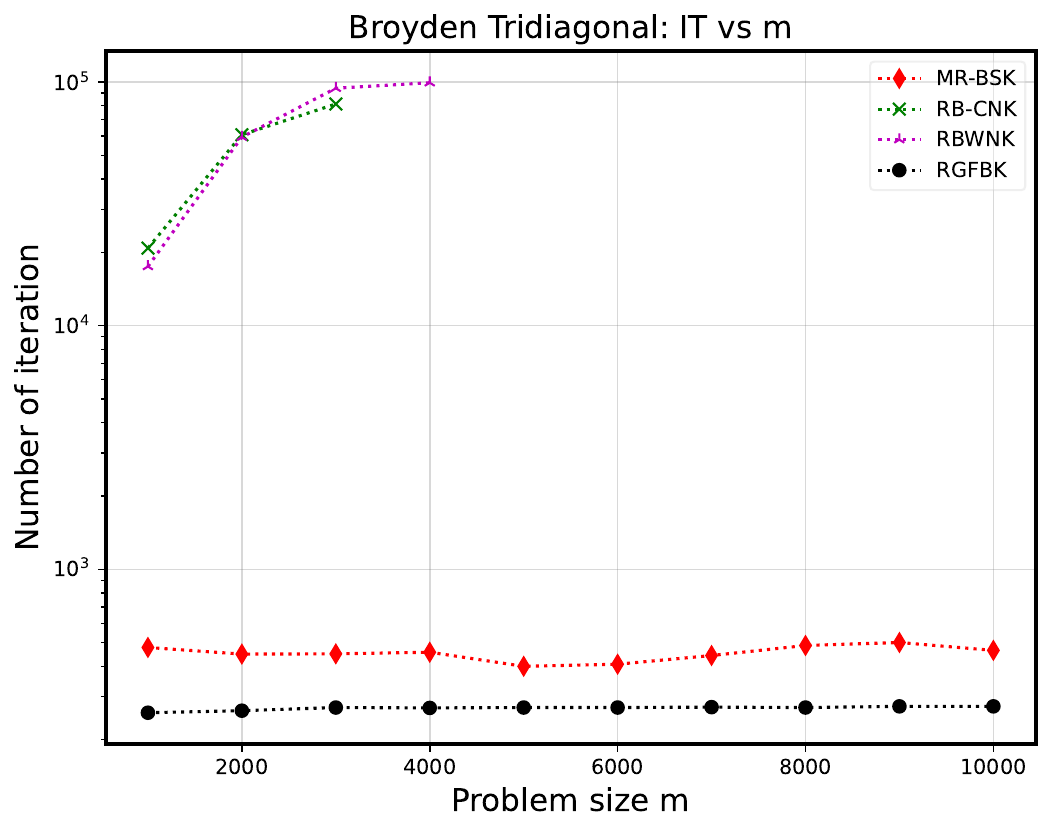}
	\hspace{0.1cm}
	\subfloat
	\centering
	\includegraphics[width=0.35\textwidth]{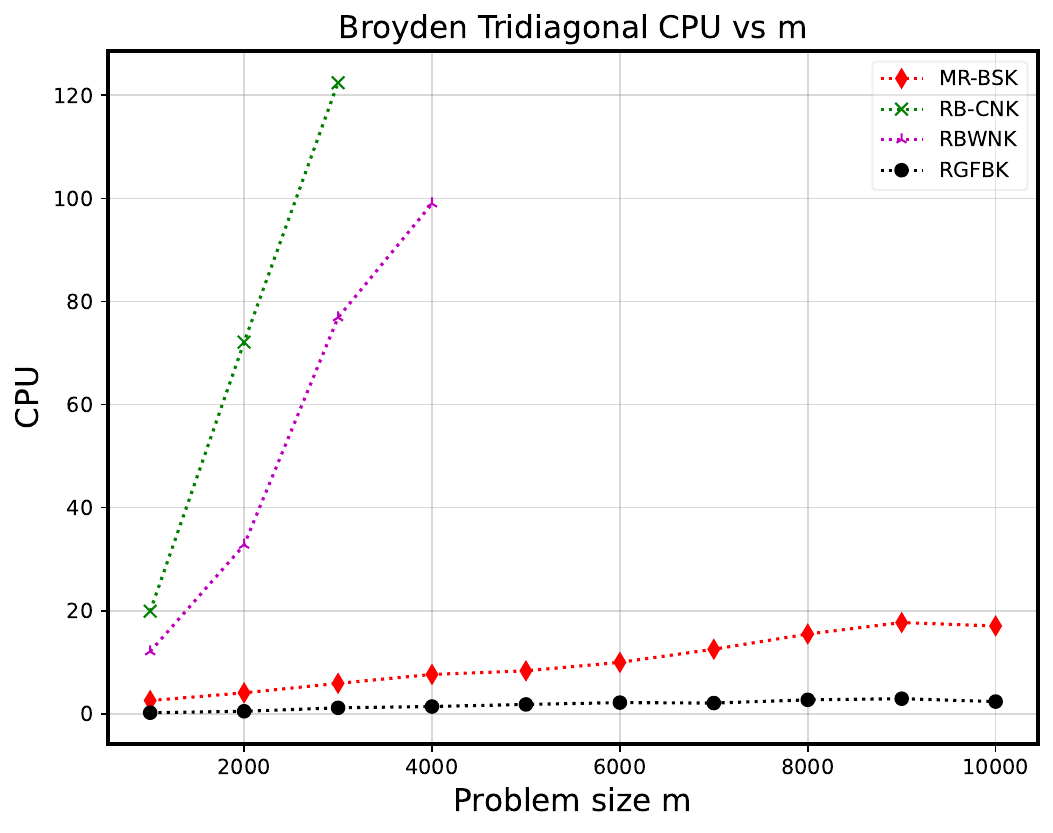}
	\hspace{0.1cm}
	\caption{Evolution of IT and CPU with the change of problem scale m, Broyden Tridiagonal Function}
	\label{fig:btm}
\end{figure}
	
\begin{figure}[!htb]
	\centering
	\subfloat
	\centering
	\includegraphics[width=0.35\textwidth]{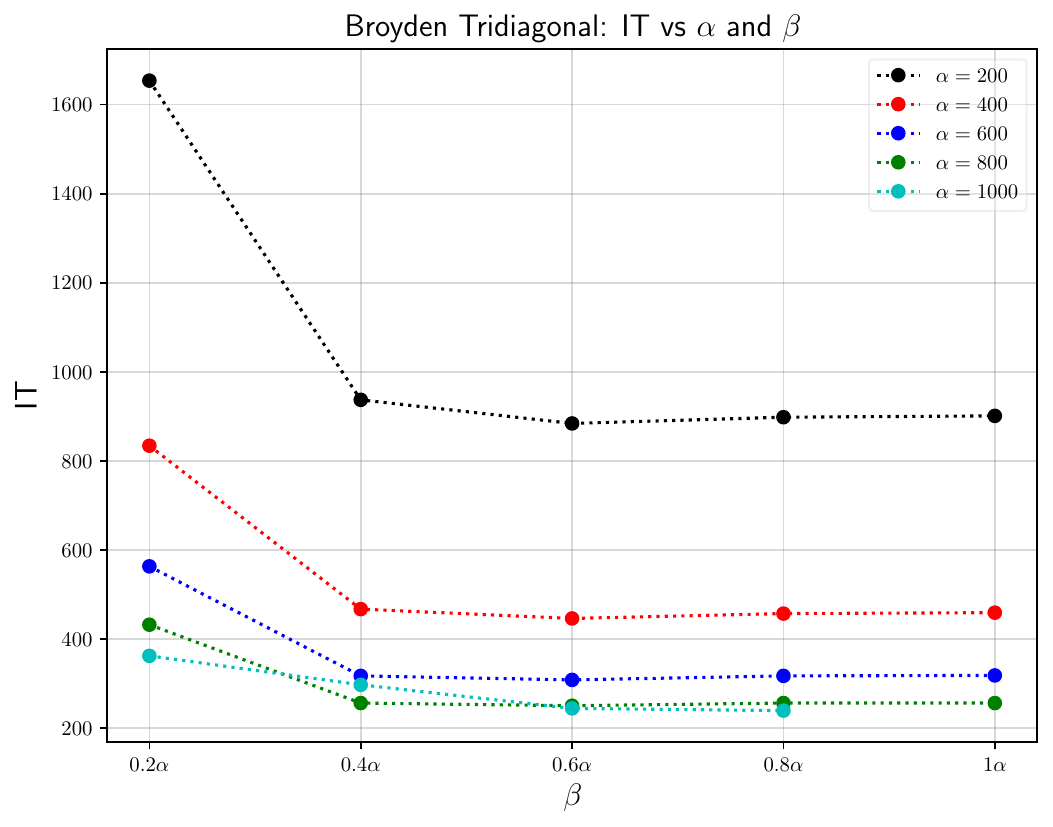}
	\hspace{0.1cm}
	\subfloat
	\centering
	\includegraphics[width=0.35\textwidth]{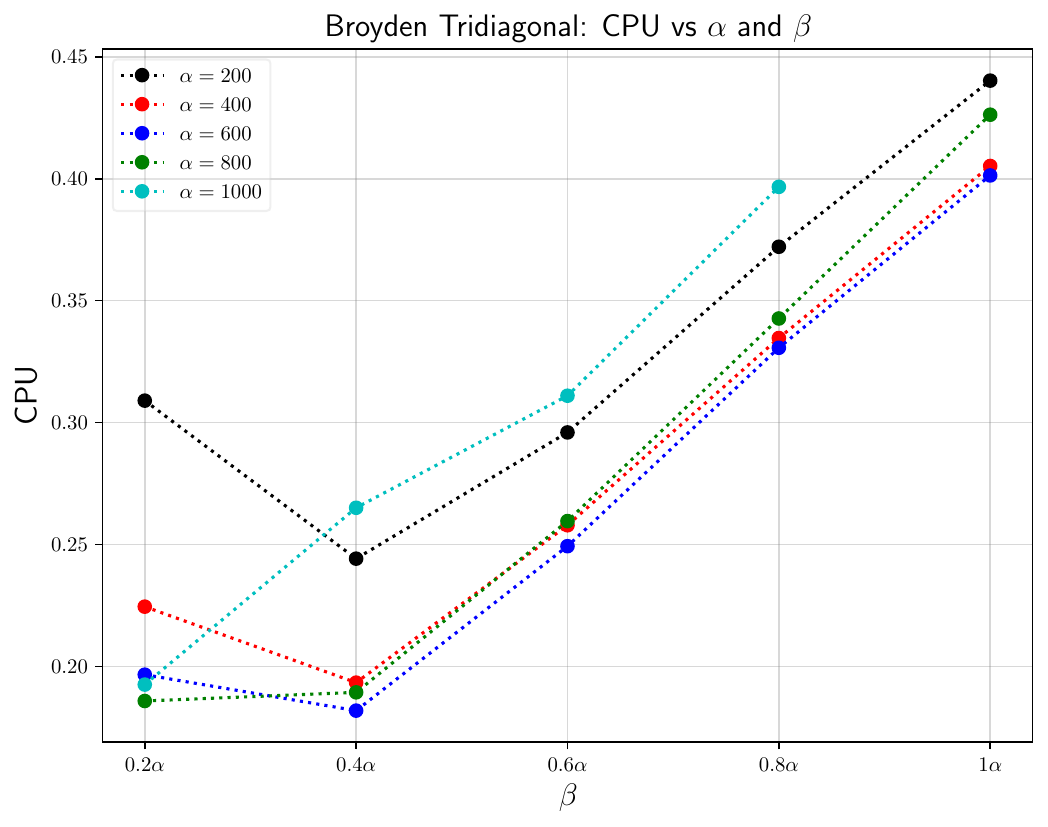}
	\hspace{0.1cm}
	
	\caption{Numerical results for the Broyden Tridiagonal Function, $m=10^3$}
	\label{fig:btp}
\end{figure}

\section{Concluding remarks}
In this paper, we propose a random greedy fast block Kaczmarz method and derive an upper bound for the convergence rate of the method, which is related to the stochastic condition number $\kappa_{\alpha,\, \beta}$ and the relaxation parameter \(\gamma\). When a favorable Jacobian matrix partitioning is obtained through random sampling of the stochastic greedy condition number and the relaxation parameter is appropriately chosen, the proposed RGFBK method can achieve faster convergence. Numerical experiments further validate the effectiveness of the proposed RGFBK method, demonstrating that it outperforms not only block algorithms based on pseudoinverse updates, including sampling-based MR-BSNK and MD-BSNK, and the greedy strategy-based RB-CNK method, but also deterministic block algorithms without pseudoinverse updates such as the RBWNK method.

\bibliographystyle{alpha}
\bibliography{NCH_ETD2}

\newcommand{\etalchar}[1]{$^{#1}$}
\begin{thebibliography}{WLBG22}

\bibitem[Atk92]{atkinson1992survey}
Kendall~E Atkinson.
\newblock A survey of numerical methods for solving nonlinear integral
  equations.
\newblock {\em The Journal of Integral Equations and Applications}, pages
  15--46, 1992.

\bibitem[BDD{\etalchar{+}}10]{babajee2010analysis}
DKR Babajee, Muhammad~Zaid Dauhoo, Mohammad~Taghi Darvishi, A~Karami, and Ali
  Barati.
\newblock Analysis of two chebyshev-like third order methods free from second
  derivatives for solving systems of nonlinear equations.
\newblock {\em Journal of Computational and Applied Mathematics},
  233(8):2002--2012, 2010.

\bibitem[BW18]{bai2018relaxed}
Zhong-Zhi Bai and Wen-Ting Wu.
\newblock On relaxed greedy randomized kaczmarz methods for solving large
  sparse linear systems.
\newblock {\em Applied Mathematics Letters}, 83:21--26, 2018.

\bibitem[CH19]{chen2019homotopy}
Qipin Chen and Wenrui Hao.
\newblock A homotopy training algorithm for fully connected neural networks.
\newblock {\em Proceedings of the Royal Society A}, 475(2231):20190662, 2019.

\bibitem[DSS20]{du2020randomized}
Kui Du, Wu-Tao Si, and Xiao-Hui Sun.
\newblock Randomized extended average block kaczmarz for solving least squares.
\newblock {\em SIAM Journal on Scientific Computing}, 42(6):A3541--A3559, 2020.

\bibitem[GC25]{gao2025convergence}
Yu~Gao and Chong Chen.
\newblock Convergence analysis of the nonlinear kaczmarz method for systems of
  nonlinear equations with componentwise convex mappings and applications to
  image reconstruction in multispectral ct.
\newblock {\em SIAM Journal on Imaging Sciences}, 18(1):120--151, 2025.

\bibitem[HLS07]{haltmeier2007kaczmarz}
Markus Haltmeier, Antonio Leitao, and Otmar Scherzer.
\newblock Kaczmarz methods for regularizing nonlinear ill-posed equations i:
  Convergence analysis.
\newblock {\em Inverse Problems and Imaging}, 1(2):289, 2007.

\bibitem[Kel99]{kelley1999iterative}
Carl~T Kelley.
\newblock {\em Iterative methods for optimization}.
\newblock SIAM, 1999.

\bibitem[Kel03]{kelley2003solving}
Carl~T Kelley.
\newblock {\em Solving nonlinear equations with Newton's method}.
\newblock SIAM, 2003.

\bibitem[Luk94]{lukvsan1994inexact}
Ladislav Luk{\v{s}}an.
\newblock Inexact trust region method for large sparse systems of nonlinear
  equations.
\newblock {\em Journal of Optimization Theory and Applications},
  81(3):569--590, 1994.

\bibitem[SV09]{strohmer2009randomized}
Thomas Strohmer and Roman Vershynin.
\newblock A randomized kaczmarz algorithm with exponential convergence.
\newblock {\em Journal of Fourier Analysis and Applications}, 15(2):262--278,
  2009.

\bibitem[TH24]{tan2024nonlinear}
Yun-Xia Tan and Zheng-Da Huang.
\newblock On a nonlinear fast deterministic block kaczmarz method for solving
  nonlinear equations.
\newblock {\em Communications on Applied Mathematics and Computation}, pages
  1--16, 2024.

\bibitem[WLBG22]{wang2022nonlinear}
Qifeng Wang, Weiguo Li, Wendi Bao, and Xingqi Gao.
\newblock Nonlinear kaczmarz algorithms and their convergence.
\newblock {\em Journal of Computational and Applied Mathematics}, 399:113720,
  2022.

\bibitem[XGY25]{xiao2025fast}
Aqin Xiao, Xiangyu Gao, and Jun-Feng Yin.
\newblock A fast block nonlinear bregman-kaczmarz method with averaging for
  nonlinear sparse signal recovery: A. v et al.
\newblock {\em Journal of Scientific Computing}, 104(2):71, 2025.

\bibitem[XY24]{xiao2024averaging}
A-Qin Xiao and Jun-Feng Yin.
\newblock On averaging block kaczmarz methods for solving nonlinear systems of
  equations.
\newblock {\em Journal of Computational and Applied Mathematics}, 451:116041,
  2024.

\bibitem[YLG22]{yuan2022sketched}
Rui Yuan, Alessandro Lazaric, and Robert~M Gower.
\newblock Sketched newton--raphson.
\newblock {\em SIAM Journal on Optimization}, 32(3):1555--1583, 2022.

\bibitem[YY24]{ye2024residual}
Yu-Xin Ye and Jun-Feng Yin.
\newblock A residual-based weighted nonlinear kaczmarz method for solving
  nonlinear systems of equations.
\newblock {\em Computational and Applied Mathematics}, 43(5):276, 2024.

\bibitem[ZL23]{zhang2023randomized}
Yanjun Zhang and Hanyu Li.
\newblock Randomized block subsampling kaczmarz-motzkin method.
\newblock {\em Linear Algebra and its Applications}, 667:133--150, 2023.

\bibitem[ZL24]{zhang2024greedy}
Yanjun Zhang and Hanyu Li.
\newblock Greedy capped nonlinear kaczmarz methods.
\newblock {\em Journal of Computational and Applied Mathematics}, 451:116067,
  2024.

\bibitem[ZLT24]{zhangli2024greedy}
Yanjun Zhang, Hanyu Li, and Ling Tang.
\newblock Greedy randomized sampling nonlinear kaczmarz methods.
\newblock {\em Calcolo}, 61(2):25, 2024.

\bibitem[ZWZ23]{zhang2023maximum}
Jianhua Zhang, Yuqing Wang, and Jing Zhao.
\newblock On maximum residual nonlinear kaczmarz-type algorithms for large
  nonlinear systems of equations.
\newblock {\em Journal of Computational and Applied Mathematics}, 425:115065,
  2023.

\end{thebibliography}
\end{document}